\documentclass[12pt]{amsart}

\usepackage[a4paper,margin=3cm]{geometry}
\usepackage[T1]{fontenc}
\usepackage[utf8]{inputenc}
\usepackage{lmodern}
\usepackage{amsmath,amssymb,amsthm}
\usepackage{mathrsfs}
\usepackage{hyperref}

\newcommand{\Z}{\mathbb{Z}}

\newcommand{\PP}{\mathcal{P}}
\newcommand{\C}{\mathbb{C}}

\newtheorem{lem}{Lemma}
\newtheorem{thm}{Theorem}
\newtheorem{prop}{Proposition}
\newtheorem{cor}{Corollary}

\theoremstyle{definition}
\newtheorem{defn}{Definition}
\newtheorem{rem}{Remark}
\newtheorem{exa}{Example}

\title{Generating functions for compositions with constrained even parts}

\author{Mahdi Koutchoukali}
\address{Aix-Marseille Univ, CNRS, Institut de Mathématiques de Marseille,
	Marseille, France}

\email{koutchoukali.mahdi@hotmail.fr}

\date{\today}

\begin{document}
	
	\maketitle
	\begin{abstract}
		We study compositions of a positive integer $n$ in which the occurrence of even parts
		larger than a fixed threshold $k$ is controlled.
		More precisely, for each composition $m=(m_1,\dots,m_r)$ we consider the number of even
		parts strictly larger than $k$, and we introduce a two-variable generating function that
		encodes this statistic.
		
		We show that this generating function is rational and obtain explicit closed forms,
		depending on the parity of $k$.
		We obtain a rational bivariate generating function that exactly enumerates
		compositions of $n$ according to the number of even parts greater than $k$.
		From it, we derive linear recurrences and exact formulas for several natural
		specializations, including compositions containing no such part, the parity
		distribution of their number, their total number of occurrences among all
		compositions of $n$, and positional statistics describing the first occurrence
		of such a part.
		
		This combinatorial problem is motivated by questions arising from combinatorial
		expansions related to zeta functions of algebraic curves over finite fields, although the
		results of this paper are entirely combinatorial.
	\end{abstract}

	\medskip
	
	\noindent\textbf{Mathematics Subject Classification (2020).}
	Primary: 05A17, 11P83; Secondary: 05A15, 05A19.

	\medskip
	
	\noindent\textbf{Keywords.}
	Integer compositions, generating functions, arithmetic constraints.
	
	% ---- Début du papier ----
	
	\section{Introduction}
	
	A composition of a positive integer $n$ is a finite ordered tuple
	$m=(m_1,\dots,m_r)\in\Z_{>0}^r$ such that $m_1+\cdots+m_r=n$.
	We denote by $\PP(n)$ the set of compositions of $n$ (with the convention
	$\PP(0)=\{\emptyset\}$).
	It is classical (see, e.g., \cite{bero}) that
	\[
	\#\PP(n)=2^{n-1}\qquad(n\ge1).
	\]
	
	In many situations, one is led to count compositions under restrictions on the
	allowed part sizes. A powerful and elementary tool is the generating function
	identity
	\[
	C_A(x):=1+\sum_{n\ge1} c_A(n)x^n=\frac{1}{1-\sum_{j\in A}x^j},
	\]
	where $c_A(n)$ denotes the number of compositions of $n$ with all parts in
	a fixed set $A\subset\Z_{>0}$; see for instance \cite{bero}.
	
	Generating function techniques play a central role in modern analytic combinatorics; we refer to the monograph of Flajolet and Sedgewick~\cite{flse} for a comprehensive and systematic treatment of these methods.

	The study of integer compositions subject to various constraints has a long and active history in enumerative and analytic combinatorics. Many authors have investigated exact enumeration problems for compositions with restricted parts or prescribed structural properties, while others have focused on asymptotic and probabilistic aspects of random compositions. For instance, compositions with constraints on the parity or the size of parts, with bounded or forbidden part sizes, or with restrictions on multiplicities and distinctness of parts, have been widely studied using generating function methods and analytic techniques (see, e.g., \cite{bahi,bero,bo,hema07,hema,hist,hota,hota25,knmu,si}). These works illustrate the richness of the subject and the effectiveness of generating functions in deriving exact formulas, recurrences, and asymptotic estimates. The present paper fits into this general line of research, but focuses on a different statistic, namely the number of even parts exceeding a given threshold.
	
	In this paper, we study a closely related, but slightly different, problem:
	we do not forbid a set of parts entirely, but rather keep track of how often
	a given arithmetic family of parts occurs.
	More precisely, fixing an integer $k\ge1$, we consider the set of even integers
	strictly larger than $k$,
	\[
	\mathrm{Ev}_{>k}:=\{j\in\Z_{>0} : j\ \text{even and}\ j>k\},
	\]
	and we study the distribution of the number of occurrences of parts in $\mathrm{Ev}_{>k}$
	among compositions of $n$.
	
	The resulting two--variable generating function turns out to be rational, and
	this gives access to explicit coefficient formulas and asymptotic estimates.\\
	
	Our original motivation comes from arithmetic questions related to zeta functions
	of algebraic curves over finite fields. Coefficients of the corresponding
	\(L\)-polynomials admit composition expansions of the type recalled below
	(see \cite{ko}). In the tower that motivated this work, their sign structure
	naturally singles out even parts above a level-dependent threshold. This leads
	to four combinatorial questions studied here: absence, parity, multiplicity,
	and position of such parts. We now state precisely the arithmetic sign pattern
	from which these questions arose.
	
	\medskip

	\subsection*{Arithmetic origin of the four statistics}
	
	We briefly recall the arithmetic setting that originally led to the
	combinatorial questions considered in this paper. The arithmetic assertions in
	the next two paragraphs are results of a separate work of the authors which has
	not yet appeared. We state them here solely as unpublished motivating input,
	not as results of the present paper. No theorem, proposition, or corollary
	proved below depends on them: the combinatorial development starts from the
	explicit sign pattern \eqref{eq:motivating_sign_pattern}.
	
	Our motivating example is the Garcia--Stichtenoth tower descended to
	\(\mathbb F_q\). At level \(k\), let \(\mathcal G_k/\mathbb F_q\) denote the
	corresponding function field and let
	\[
	L_{\mathcal G_k}(t)=\sum_{n=0}^{2g_k}a_n^{(k)}t^n
	\]
	be its \(L\)-polynomial. The explicit composition formula for its coefficients
	has the form
	\[
	a_n^{(k)}
	=
	\sum_{m=(m_1,\ldots,m_r)\in\PP(n)} CR_k(m),
	\qquad
	CR_k(m)
	:=
	\prod_{s=1}^{r}
	\frac{S_{m_s}(\mathcal G_k)}{m_1+\cdots+m_s},
	\]
	where
	\[
	S_i(\mathcal G_k)=N_i(\mathcal G_k)-(q^i+1).
	\]
	
	In the separate arithmetic analysis mentioned above, one obtains, for every
	fixed \(i\),
	\[
	N_i(\mathcal G_k)=
	\begin{cases}
		N_2(\mathcal G_k)+o(q^k),& i\ \text{even},\\
		o(q^k),& i\ \text{odd},
	\end{cases}
	\qquad (k\to\infty),
	\]
	together with
	\[
	N_2(\mathcal G_k)
	=
	(q^2-1)q^{k-1}+O(1)
	=
	q^{k+1}-q^{k-1}+O(1).
	\]
	The arithmetic situation motivating the present paper is the case in which
	\[
	N_i(\mathcal G_k)<q^i+1
	\qquad\text{for every positive odd integer }i.
	\]
	In that setting, for \(k\) sufficiently large, the analysis leads to the sign
	pattern
	\begin{equation}\label{eq:motivating_sign_pattern}
		S_i(\mathcal G_k)
		\begin{cases}
			<0,& i\ge k+1,\ i\ \text{even},\\
			<0,& i\ge k+1,\ i\ \text{odd},\\
			>0,& i<k+1,\ i\ \text{even},\\
			<0,& i<k+1,\ i\ \text{odd}.
		\end{cases}
	\end{equation}
	We take this sign pattern as the starting point of the combinatorial problem
	studied in the present paper. Equivalently, even indices \(i\le k\) give
	positive factors, whereas odd indices and even indices \(i>k\) give negative
	factors. Thus the threshold \(k\) is precisely the arithmetic threshold at
	which the sign of the even-indexed quantities \(S_i(\mathcal G_k)\) changes.
	
	This leads naturally to the following statistic. For a composition
	\(m=(m_1,\ldots,m_r)\in\PP(n)\), set
	\[
	B(m):=\#\{\,s\in\{1,\ldots,r\}:m_s\text{ is even and }m_s>k\,\}.
	\]
	The next elementary observation is the precise bridge between the arithmetic
	sign pattern and the combinatorial statistic \(B(m)\).
	
	\begin{lem}\label{lem:arithmetic_sign_bridge}
		Assume \eqref{eq:motivating_sign_pattern} and let
		\(m=(m_1,\ldots,m_r)\in\PP(n)\). Then
		\[
		\operatorname{sgn}\bigl(CR_k(m)\bigr)=(-1)^{n+B(m)}.
		\]
		Consequently, if \(n\) is even, \(CR_k(m)\) is positive precisely when
		\(B(m)\) is even, whereas if \(n\) is odd, \(CR_k(m)\) is positive precisely
		when \(B(m)\) is odd.
	\end{lem}
	
	\begin{proof}
		All denominators in \(CR_k(m)\) are positive. By
		\eqref{eq:motivating_sign_pattern}, a negative numerator occurs for every odd
		part and for every even part strictly larger than \(k\). If \(o(m)\) denotes
		the number of odd parts of \(m\), then
		\[
		\operatorname{sgn}\bigl(CR_k(m)\bigr)=(-1)^{o(m)+B(m)}.
		\]
		Since the parity of the sum of the parts is the parity of the number of odd
		parts, \(o(m)\equiv n\pmod 2\), and the result follows.
	\end{proof}
	
	Lemma~\ref{lem:arithmetic_sign_bridge} leads to four increasingly refined
	combinatorial questions. First, compositions containing no even part beyond
	the threshold form a canonical same-sign subfamily; their enumeration is
	studied in Section~\ref{sec:avoid}. Second, the parity of the number of such
	parts determines the sign of each contribution, which motivates
	Section~\ref{sec:parity}. Third, the weighted arithmetic problem also makes it
	natural to measure how frequently these sign-changing even factors occur;
	their total multiplicity and average number are studied in
	Section~\ref{sec:total}.
	
	Finally, the absolute weight
	\[
	|CR_k(m)|
	=
	\prod_{s=1}^{r}
	\frac{|S_{m_s}(\mathcal G_k)|}{m_1+\cdots+m_s}
	\]
	depends on the positions of the parts through the partial-sum denominators.
	If a large even part \(m_i\) is modified in an attempt to compare
	contributions of opposite signs, the preceding partial sum
	\[
	D=m_1+\cdots+m_{i-1}
	\]
	enters naturally. Thus the location of the first large even part contains
	information that is lost when one records only its existence, parity, or
	multiplicity. This motivates the positional statistics studied in
	Section~\ref{sec:positional}. Notice that the frequency information developed
	here does not by itself compare the total absolute weights of the positive and
	negative contributions; the latter is a separate arithmetic problem.
	
	For emphasis, the unpublished arithmetic assertions above are never invoked
	after this motivational subsection except through references back to the
	motivation; all enumerative statements from Section~\ref{sec:notations}
	onward are proved from the definitions and the displayed generating functions.
	
	\medskip
	
	\section{Notations and basic definitions}\label{sec:notations}
	
	A \emph{composition} of a nonnegative integer $n$ is a finite sequence $m=(m_1,\dots,m_r)$ of positive integers such that
	\[
	m_1+\cdots+m_r = n.
	\]
	We denote by $\PP(n)$ the set of all compositions of $n$, with the convention that $\PP(0)=\{\varnothing\}$ consists of the empty composition. It is classical (see, e.g., \cite{bero}) that $\#\PP(n)=2^{n-1}$ for $n\ge 1$.
	
	Throughout the paper, we fix an integer $k\ge 1$. We denote by
	\[
	\mathrm{Ev}_{>k} = \{ j \in \mathbb{Z}_{\ge 1} : j \text{ is even and } j>k \}
	\]
	the set of even integers strictly greater than $k$. For a composition $m=(m_1,\dots,m_r)$, we define
	\[
	B(m) := \#\{\, t \in \{1,\dots,r\} : m_t \in \mathrm{Ev}_{>k} \,\},
	\]
	that is, $B(m)$ is the number of parts of $m$ that are even and strictly greater than $k$.
	
	We introduce the bivariate generating function
	\[
	F(x,y) := \sum_{n\ge 0} \left( \sum_{m\in \PP(n)} y^{B(m)} \right) x^n.
	\]
	In other words, the variable $x$ marks the size of the composition, while the variable $y$ keeps track of the number of parts counted by $B(m)$.
	
	For a formal power series
	\[
	A(x) = \sum_{n\ge 0} a_n x^n,
	\]
	we use the standard coefficient extraction notation
	\[
	[x^n]\,A(x) := a_n,
	\]
	that is, $[x^n]\,A(x)$ denotes the coefficient of $x^n$ in the series expansion of $A(x)$.
	
	Accordingly, we write
	\[
	F(x,y) = \sum_{n\ge 0} f_n(y)\, x^n, \qquad \text{where } f_n(y) := [x^n]\,F(x,y).
	\]
	By definition, we have
	\[
	f_n(y) = \sum_{m\in \PP(n)} y^{B(m)}.
	\]
	Expanding $f_n(y)$ as a polynomial in $y$, we write
	\[
	f_n(y) = \sum_{t\ge 0} a_{n,t}\, y^t,
	\]
	where
	\[
	a_{n,t} = \#\{\, m \in \PP(n) : B(m) = t \,\}
	\]
	is the number of compositions of $n$ having exactly $t$ parts that are even and strictly greater than $k$.
	
	\medskip

	For a function $G(x,y)$ that is differentiable with respect to $y$, we denote by
	\[
	\frac{\partial}{\partial y} G(x,y)
	\]
	its partial derivative with respect to the variable $y$.
	
	For a function $G(x,y)$, we denote by $\left.G(x,y)\right|_{y=1}$ its value after
	specialization at $y=1$, and similarly $\left.\dfrac{\partial}{\partial y}G(x,y)\right|_{y=1}$
	for the partial derivative evaluated at $y=1$.

	\section{The main generating function}
	
	We first establish a closed form expression for the bivariate generating function
	$F(x,y)$ introduced in the previous section.
	
	\begin{prop}\label{prop:Fxy_rational}
		Let $k\ge 1$ be fixed. Then the generating function
		\[
		F(x,y) = \sum_{n\ge 0} \left( \sum_{m\in \PP(n)} y^{B(m)} \right) x^n
		\]
		is a rational function in $x$ and $y$. More precisely, one has
		\[
		F(x,y) = \frac{1}{1 - \sum_{j\ge 1} w_j x^j},
		\]
		where the weights $w_j$ are given by
		\[
		w_j :=
		\begin{cases}
			y, & \text{if } j \in \mathrm{Ev}_{>k},\\
			1, & \text{otherwise}.
		\end{cases}
		\]
		Equivalently,
		\[
		F(x,y) = \frac{1}{1 - \sum_{j\ge 1} x^j - (y-1)\sum_{j\in \mathrm{Ev}_{>k}} x^j }.
		\]

	\end{prop}

	\begin{proof}
		Any composition $m\in\PP(n)$ is a finite ordered tuple
		\[
		m=(m_1,\dots,m_r)
		\]
		of positive integers, where the integer $r\ge 1$ is the number of parts; we also
		allow the empty composition $\emptyset\in\PP(0)$, which corresponds to the case
		$r=0$. We construct the generating function by summing the contributions of
		compositions according to their length $r$.
		
		For each $j\ge 1$, define
		\[
		w_j =
		\begin{cases}
			y, & \text{if } j\in \mathrm{Ev}_{>k},\\
			1, & \text{otherwise}.
		\end{cases}
		\]
		A part of size $j$ contributes a factor $x^j$ to record the total sum, and an
		additional factor $w_j$ to record whether it belongs to $\mathrm{Ev}_{>k}$. Therefore, the
		weight of a composition $m=(m_1,\dots,m_r)$ is
		\[
		(w_{m_1}x^{m_1})\cdots(w_{m_r}x^{m_r})
		= x^{m_1+\cdots+m_r}\, y^{B(m)}.
		\]
		Summing over all possible sizes of a single part yields the generating function
		\[
		A(x,y) := \sum_{j\ge 1} w_j x^j.
		\]
		Although this sum runs over infinitely many values of $j$, it still corresponds
		to the choice of a single part whose size may be any positive integer.
		
		A composition with exactly $r$ parts is an ordered $r$--tuple of such parts, and
		the choices of the parts are independent. By the product rule for generating
		functions, the generating function for compositions of length $r$ is therefore
		$A(x,y)^r$. Summing over all $r\ge 0$, where $r=0$ corresponds to the empty
		composition, gives
		\[
		F(x,y)=\sum_{r\ge 0} A(x,y)^r.
		\]
		Finally, in the ring of formal power series in $x$ with coefficients in $\Z[y]$,
		the geometric series identity holds:
		\[
		(1-A(x,y))\left(\sum_{r\ge 0} A(x,y)^r\right)=1,
		\]
		since the product expands as
		\[
		(1+A+A^2+\cdots)-(A+A^2+A^3+\cdots)=1,
		\]
		and all other terms cancel. Hence $\sum_{r\ge 0} A(x,y)^r$ is the multiplicative
		inverse of $1-A(x,y)$, and we obtain
		\[
		F(x,y)=\frac{1}{1-A(x,y)}=\frac{1}{1-\sum_{j\ge 1} w_j x^j},
		\]
		which proves the claim.
		
		To get the equivalent form, we split the sum defining $A(x,y)$ according to
		whether $j$ belongs to $\mathrm{Ev}_{>k}$ or not:
		\[
		\sum_{j\ge 1} w_j x^j
		= \sum_{j\notin \mathrm{Ev}_{>k}} 1\cdot x^j + \sum_{j\in \mathrm{Ev}_{>k}} y\,x^j
		= \sum_{j\ge 1} x^j + (y-1)\sum_{j\in \mathrm{Ev}_{>k}} x^j.
		\]
		Substituting this identity into the previous expression yields
		\[
		F(x,y)=\frac{1}{1-\sum_{j\ge 1} x^j-(y-1)\sum_{j\in \mathrm{Ev}_{>k}} x^j},
		\]
		as claimed.
	\end{proof}
	\subsection{An explicit rational form}
	
	We now make the previous expression more explicit by evaluating the geometric
	series $\sum_{j\in \mathrm{Ev}_{>k}} x^j$. This will be useful later on to derive recurrences
	for the coefficients of $F(x,y)$ and to study specializations such as $y=0$ and
	$y=-1$.
	
	\begin{prop}\label{prop:sum_even_gt_k}
		Let $k\ge 1$ and set
		\[
		\delta :=
		\begin{cases}
			1, & \text{if $k$ is odd},\\
			2, & \text{if $k$ is even}.
		\end{cases}
		\]
		Then
		\[
		\sum_{j\in \mathrm{Ev}_{>k}} x^j \;=\; \sum_{m\ge 0} x^{k+\delta+2m}
		\;=\; \frac{x^{k+\delta}}{1-x^2}.
		\]
	\end{prop}
	
	\begin{proof}
		The set $\mathrm{Ev}_{>k}$ consists of the even integers strictly larger than $k$.
		If $k$ is odd, then the smallest even integer $>k$ is $k+1$, so the elements of
		$Ev_{>k}$ are
		\[
		k+1,\ k+3,\ k+5,\ \dots
		\]
		i.e.\ the arithmetic progression $k+1+2m$ for $m\ge 0$. If $k$ is even, then the
		smallest even integer $>k$ is $k+2$, hence the elements of $\mathrm{Ev}_{>k}$ are
		\[
		k+2,\ k+4,\ k+6,\ \dots
		\]
		i.e.\ the arithmetic progression $k+2+2m$ for $m\ge 0$.
		
		Both cases are captured by the parameter $\delta\in\{1,2\}$ defined above, and we
		obtain
		\[
		\sum_{j\in \mathrm{Ev}_{>k}} x^j = \sum_{m\ge 0} x^{k+\delta+2m}
		= x^{k+\delta}\sum_{m\ge 0} (x^2)^m
		= \frac{x^{k+\delta}}{1-x^2},
		\]
		using the geometric series identity in formal power series.
	\end{proof}
	
	Combining Proposition~\ref{prop:Fxy_rational} with Proposition~\ref{prop:sum_even_gt_k},
	we obtain an explicit rational expression for $F(x,y)$.
	
	\begin{cor}\label{cor:Fxy_explicit}
		With \(\delta\) as in Proposition~\ref{prop:sum_even_gt_k}, one has
		\[
		\boxed{
			F(x,y)
			=
			\frac{1-x^2}
			{(1-2x)(1+x)-(y-1)x^{k+\delta}}
			=
			\frac{1-x^2}
			{1-x-2x^2-(y-1)x^{k+\delta}}.
		}
		\]
	\end{cor}
	
	\begin{rem}\label{rem:exact_distribution}
		The bivariate series contains the complete distribution of \(B\). For every
		\(n,j\ge0\),
		\[
		\boxed{
			\#\{m\in\PP(n):B(m)=j\}=[x^ny^j]F(x,y),
		}
		\]
		where \([x^ny^j]G(x,y)\) denotes the coefficient of \(x^ny^j\) in the formal
		power-series expansion of \(G(x,y)\). Thus the rational function above exactly
		enumerates compositions of \(n\) according to any prescribed value of the
		number of even parts strictly larger than \(k\).
	\end{rem}
	
	\subsection{A recurrence for the coefficients}
	
	Since \(F(x,y)\) is rational in \(x\), for each fixed \(y\) its coefficient
	sequence satisfies a linear recurrence with constant coefficients; this standard
	consequence of rationality is recalled, for instance, in \cite{flse}. We record
	the resulting coefficient identity for later use.
	
	\begin{prop}\label{prop:recurrence_fn}
		Let \(s_k:=k+\delta\), the smallest even integer strictly larger than \(k\).
		Set \(f_n(y):=[x^n]F(x,y)\) for \(n\ge0\), and \(f_n(y):=0\) for \(n<0\).
		Then, for every \(n\ge0\),
		\[
		f_n(y)-f_{n-1}(y)-2f_{n-2}(y)-(y-1)f_{n-s_k}(y)
		=
		\begin{cases}
			1,&n=0,\\
			-1,&n=2,\\
			0,&\text{otherwise}.
		\end{cases}
		\]
		In particular, for every \(n\ge \max\{3,s_k\}\),
		\[
		\boxed{
			f_n(y)=f_{n-1}(y)+2f_{n-2}(y)+(y-1)f_{n-s_k}(y).
		}
		\]
	\end{prop}
	
	\begin{proof}
		Corollary~\ref{cor:Fxy_explicit} gives
		\[
		\bigl(1-x-2x^2-(y-1)x^{s_k}\bigr)
		\sum_{n\ge0}f_n(y)x^n=1-x^2.
		\]
		Coefficient comparison, with \(f_j(y)=0\) for \(j<0\), yields the first
		identity. Its right-hand side vanishes for \(n\ge 3\); hence the displayed
		homogeneous recurrence holds for \(n\ge \max\{3,s_k\}\).
	\end{proof}
	
	\section{Specializations and refined enumerations}
	
	In this section, we exploit the explicit rational form of the bivariate generating
	function $F(x,y)$ to derive several refined enumerative results. We first study the
	specialization $y=0$, which corresponds to compositions avoiding large even parts
	altogether and leads to exact formulas for their number. We then turn to the
	specialization $y=-1$, which allows us to separate compositions according to the
	parity of the number of such parts. Finally, we investigate the total number of
	occurrences of large even parts among all compositions of a given integer, which
	naturally leads to an average-value interpretation.

	\subsection{Avoiding large even parts: the specialization $y=0$}\label{sec:avoid}
	
	From the arithmetic viewpoint described in the introduction, this is the
	coarsest favourable subfamily to isolate. A composition with \(B(m)=0\)
	contains no even part beyond the threshold at which the even-indexed
	\(S_i^{(k)}\) change sign. By Lemma~\ref{lem:arithmetic_sign_bridge}, all such
	compositions therefore contribute with the same sign \((-1)^n\) to the
	coefficient expansion. Thus \(c_k(n)\) counts a canonical same-sign subfamily
	of the terms occurring in \(a_n^{(k)}\).
	
	Setting $y=0$ in $F(x,y)$ amounts to counting only those compositions for which
	$B(m)=0$, i.e.\ compositions having \emph{no} part in $\mathrm{Ev}_{>k}$.
	We denote this number by
	\[
	c_k(n):=\#\{\, m\in \PP(n): B(m)=0\,\}.
	\]
	By definition of $F(x,y)$, we have
	\[
	\sum_{n\ge 0} c_k(n)\,x^n = F(x,0).
	\]
	
	\begin{prop}\label{prop:F_x_0}
		With $\delta\in\{1,2\}$ as in Proposition~\ref{prop:sum_even_gt_k}, one has
		\[
		F(x,0)=\frac{(1-x)(1-x^2)}{(1-2x)(1-x^2)+(1-x)x^{k+\delta}}.
		\]
	\end{prop}
	
	\begin{proof}
		This is immediate from Corollary~\ref{cor:Fxy_explicit}. Indeed, when $y=0$ we
		have $y-1=-1$, hence
		\[
		F(x,0)=\frac{(1-x)(1-x^2)}{(1-2x)(1-x^2)- (y-1)(1-x)x^{k+\delta}}
		=\frac{(1-x)(1-x^2)}{(1-2x)(1-x^2)+(1-x)x^{k+\delta}}.
		\]
	\end{proof}
	
	\medskip
	
	Since $F(x,0)$ is a rational function, the sequence $(c_k(n))_{n\ge 0}$ satisfies
	a linear recurrence with constant coefficients. Moreover, it admits an exact
	closed form in terms of the roots of its denominator, as we now explain.
	
	\subsubsection*{Exact formula for $c_k(n)$ via partial fractions}
	
	Let us write
	\[
	F(x,0)=\frac{(1-x)(1-x^2)}{Q_k(x)},
	\qquad
	Q_k(x):=(1-2x)(1-x^2)+(1-x)x^{k+\delta}.
	\]
	Cancelling the common factor $(1-x)$, we may equivalently write
	\[
	F(x,0)=\frac{P_k(x)}{R_k(x)},
	\qquad
	P_k(x):=1-x^2,
	\qquad
	R_k(x):=(1-2x)(1+x)+x^{k+\delta}.
	\]
	
	When the roots of \(R_k(x)\) are simple, denote them by
	\(\alpha_{k,1},\dots,\alpha_{k,d}\). The following standard partial-fraction
	formula then applies. If repeated roots occur, the usual higher-order
	partial-fraction terms must be used instead.
	
	\begin{lem}\label{lem:partial_fraction_general}
		Let
		\[
		G(x)=\frac{P(x)}{Q(x)}
		\]
		be a proper rational function, where \(P,Q\in\C[x]\),
		\(\deg P<\deg Q\), and assume that all roots
		\(\alpha_1,\dots,\alpha_d\) of \(Q(x)\) are simple. Then there exist constants
		$C_1,\dots,C_d\in\C$ such that
		\[
		[x^n]\,G(x) = \sum_{i=1}^d C_i\,\alpha_i^{-n}\qquad(n\ge 0),
		\]
		where
		\[
		C_i = -\frac{P(\alpha_i)}{\alpha_i Q'(\alpha_i)}.
		\]
	\end{lem}
	\begin{proof}
		This is the standard partial-fraction decomposition for a rational function
		with simple poles (see, e.g., \cite{flse}); the displayed constant is the
		corresponding residue coefficient.
	\end{proof}
	
	Applying Lemma~\ref{lem:partial_fraction_general} to $F(x,0)=P_k(x)/R_k(x)$ gives:
	
	\begin{thm}\label{thm:ck_exact_formula}
	Assume \(k\ge 2\). Let
	\(\alpha_{k,1},\dots,\alpha_{k,d}\) be the simple roots of
	\[
	R_k(x)=(1-2x)(1+x)+x^{k+\delta}.
	\]
		Then, for all $n\ge 0$, one has the exact formula
		\[
		\boxed{
			c_k(n)=\sum_{i=1}^d C_{k,i}\,\alpha_{k,i}^{-n},
			\qquad
			C_{k,i}=-\frac{1-\alpha_{k,i}^2}{\alpha_{k,i} R_k'(\alpha_{k,i})}.
		}
		\]
	\end{thm}
	
	If \(R_k\) has repeated roots, the standard higher-order partial-fraction
	decomposition (see \cite{flse}) gives the corresponding exact formula, with
	polynomial factors in \(n\) multiplying powers of the reciprocal roots. Thus
	the simple-root hypothesis is used only to display the particularly compact
	form above.
	
	This formula expresses $c_k(n)$ \emph{exactly} as a finite linear combination of
	exponential terms $\alpha^{-n}$, and allows effective numerical computation for
	any fixed $k$ and $n$.
	
	The case \(k=1\) is elementary and can be treated separately. Indeed,
	\[
	F(x,0)
	=
	\frac{1-x^2}{1-x-x^2}
	=
	1+\frac{x}{1-x-x^2}.
	\]
	Thus the coefficients are obtained directly from the Fibonacci
	recurrence. The restriction \(k\ge2\) above concerns only the compact
	proper-fraction representation.

	\subsection{Even versus odd number of large even parts}\label{sec:parity}
	
	This specialization is the most direct one from the motivating arithmetic
	problem. Lemma~\ref{lem:arithmetic_sign_bridge} shows that the parity of
	\(B(m)\) determines the sign of the individual contribution \(CR_k(m)\).
	Consequently, \(E_k(n)\) and \(O_k(n)\) count the two sign classes in the
	composition expansion (with their roles interchanged according to the parity
	of \(n\)), while \(E_k(n)-O_k(n)\) is their signed numerical imbalance.
	
	Recall that
	\[
	f_n(y)=[x^n]F(x,y)=\sum_{t\ge 0} a_{n,t}y^t,
	\qquad
	a_{n,t}=\#\{\,m\in\PP(n): B(m)=t\,\}.
	\]
	We define
	\[
	E_k(n):=\#\{\,m\in\PP(n): B(m)\ \text{is even}\,\},
	\qquad
	O_k(n):=\#\{\,m\in\PP(n): B(m)\ \text{is odd}\,\}.
	\]
	Then
	\[
	E_k(n)+O_k(n)=\sum_{t\ge 0} a_{n,t}=f_n(1),
	\qquad
	E_k(n)-O_k(n)=\sum_{t\ge 0} a_{n,t}(-1)^t=f_n(-1).
	\]
	Solving gives
	\[
	E_k(n)=\frac{f_n(1)+f_n(-1)}{2},
	\qquad
	O_k(n)=\frac{f_n(1)-f_n(-1)}{2}.
	\]
	
	Since
	\[
	F(x,1)=\frac{1-x}{1-2x},
	\qquad
	[x^n]F(x,1)=2^{n-1}\quad(n\ge 1),
	\]
	we only need an explicit form for $[x^n]F(x,-1)$.
	
	\begin{prop}\label{prop:F_x_minus1}
		One has
		\[
		F(x,-1)=\frac{(1-x)(1-x^2)}{\widetilde Q_k(x)},
		\qquad
		\widetilde Q_k(x)=(1-2x)(1-x^2)+2(1-x)x^{k+\delta}.
		\]
		After cancelling $(1-x)$, this can be written as
		\[
		F(x,-1)=\frac{1-x^2}{\widetilde R_k(x)},
		\qquad
		\widetilde R_k(x)=(1-2x)(1+x)+2x^{k+\delta}.
		\]
	\end{prop}
	
	\begin{proof}
		This follows from Corollary~\ref{cor:Fxy_explicit} by specializing $y=-1$.
	\end{proof}
	
	Assume \(k\ge 2\) and that the roots of
	\[
	\widetilde R_k(x)=(1-2x)(1+x)+2x^{k+\delta}
	\]
	are simple, and denote them by
	\(\beta_{k,1},\dots,\beta_{k,e}\).
	Applying Lemma~\ref{lem:partial_fraction_general} again, we obtain:
	
	\begin{thm}\label{thm:EkOk_exact}
		Under the preceding hypotheses, for all $n\ge 1$ one has the exact formulas
		\[
		\boxed{
			E_k(n)=2^{n-2}+\frac12\sum_{j=1}^e \widetilde C_{k,j}\,\beta_{k,j}^{-n},
			\qquad
			O_k(n)=2^{n-2}-\frac12\sum_{j=1}^e \widetilde C_{k,j}\,\beta_{k,j}^{-n},
		}
		\]
		where
		\[
		\widetilde C_{k,j}=-\frac{1-\beta_{k,j}^2}{\beta_{k,j}\widetilde R_k'(\beta_{k,j})}.
		\]
	\end{thm}
	
	If \(\widetilde R_k\) has repeated roots, the same coefficient extraction is
	obtained from the standard higher-order partial-fraction decomposition (see
	\cite{flse}); the corresponding terms are polynomials in \(n\) multiplied by
	powers of the reciprocal roots. No simple-root assertion is needed elsewhere
	in the paper.
	
	For \(k=1\), the specialization is particularly simple:
	\[
	F(x,-1)
	=
	\frac{1-x^2}{1-x}
	=
	1+x.
	\]
	Consequently,
	\[
	E_1(1)=1,\qquad O_1(1)=0,
	\]
	and, for every \(n\ge2\),
	\[
	E_1(n)=O_1(n)=2^{n-2}.
	\]
	
	These identities give \emph{exact closed forms} for the number of compositions of
	$n$ with an even (resp.\ odd) number of even parts strictly larger than $k$.
	
	The preceding formulas now return directly to the arithmetic motivation.
	Define
	\[
	P_k^+(n):=\#\{m\in\PP(n):CR_k(m)>0\},
	\qquad
	P_k^-(n):=\#\{m\in\PP(n):CR_k(m)<0\}.
	\]
	
	\begin{cor}\label{cor:arithmetic_sign_imbalance}
		Under the sign pattern \eqref{eq:motivating_sign_pattern}, one has
		\[
		(P_k^+(n),P_k^-(n))
		=
		\begin{cases}
			(E_k(n),O_k(n)),& n\ \mathrm{even},\\
			(O_k(n),E_k(n)),& n\ \mathrm{odd}.
		\end{cases}
		\]
		Consequently,
		\[
		P_k^+(n)-P_k^-(n)
		=
		(-1)^n\bigl(E_k(n)-O_k(n)\bigr)
		=
		(-1)^n[x^n]F(x,-1),
		\]
		where, as defined in Section~\ref{sec:notations}, \([x^n]G(x)\) denotes the coefficient of
		\(x^n\) in the power-series expansion of \(G(x)\).
	\end{cor}
	
	\begin{proof}
		This is an immediate consequence of
		Lemma~\ref{lem:arithmetic_sign_bridge} and the definitions of \(E_k(n)\) and
		\(O_k(n)\).
	\end{proof}
	
	Thus \(F(x,-1)\) measures exactly the numerical imbalance between the two sign
	classes in the composition expansion of the \(L\)-polynomial coefficient
	\(a_n^{(k)}\). This statement concerns the number of contributions of each
	sign, not the comparison of their absolute arithmetic weights.

	\subsection{Total number of large even parts among all compositions}\label{sec:total}
	
	The parity statistic records only whether the number of sign-changing large
	even parts is even or odd. For the weighted arithmetic problem one is also led
	to ask how many such factors are typically present. The quantity \(T_k(n)\)
	records their total multiplicity over all compositions and hence provides a
	first quantitative measure of the frequency of the factors responsible for
	changing the sign of \(CR_k(m)\). We emphasize that this statistic alone does
	not compare the total positive and negative \emph{weights}; such a comparison
	also depends on the magnitudes of the \(S_i^{(k)}\) and on the partial-sum
	denominators.
	
	In addition to the parity questions, it is natural to study the \emph{total}
	number of occurrences of parts in $\mathrm{Ev}_{>k}$ among all compositions of $n$.
	
	\begin{defn}
		For $n\ge 0$, define
		\[
		T_k(n):=\sum_{m\in\PP(n)} B(m).
		\]
		Equivalently, $T_k(n)$ counts, with multiplicity, the total number of parts
		belonging to $\mathrm{Ev}_{>k}$ across all compositions of $n$.
	\end{defn}
	
	By definition of $F(x,y)$,
	\[
	F(x,y)=\sum_{n\ge 0}\left(\sum_{m\in\PP(n)} y^{B(m)}\right)x^n,
	\]
	hence differentiating with respect to $y$ gives
	\[
	\frac{\partial}{\partial y}F(x,y)
	=\sum_{n\ge 0}\left(\sum_{m\in\PP(n)} B(m)\,y^{B(m)-1}\right)x^n.
	\]
	Specializing at $y=1$ yields the generating function of $(T_k(n))_{n\ge 0}$:
	\[
	\sum_{n\ge 0} T_k(n)\,x^n=\left.\frac{\partial}{\partial y}F(x,y)\right|_{y=1}.
	\]
	
	\begin{prop}\label{prop:Tk_genfun}
		With $\delta\in\{1,2\}$ as in Proposition~\ref{prop:sum_even_gt_k}, one has
		\[
		\sum_{n\ge 0} T_k(n)\,x^n
		=\left.\frac{\partial}{\partial y}F(x,y)\right|_{y=1}
		=\frac{x^{k+\delta}(1-x)^2}{(1-2x)^2(1-x^2)}.
		\]
		Equivalently,
		\[
		\sum_{n\ge 0} T_k(n)\,x^n
		=\frac{x^{k+\delta}(1-x)}{(1-2x)^2(1+x)}.
		\]
	\end{prop}

	\begin{proof}
		From Corollary~\ref{cor:Fxy_explicit},
		\[
		F(x,y)=
		\frac{1-x^2}
		{1-x-2x^2-(y-1)x^{k+\delta}}.
		\]
		Differentiating with respect to \(y\) gives
		\[
		\frac{\partial}{\partial y}F(x,y)
		=
		\frac{(1-x^2)x^{k+\delta}}
		{\bigl(1-x-2x^2-(y-1)x^{k+\delta}\bigr)^2}.
		\]
		At \(y=1\),
		\[
		\left.\frac{\partial}{\partial y}F(x,y)\right|_{y=1}
		=
		\frac{x^{k+\delta}(1-x^2)}
		{(1-x-2x^2)^2}.
		\]
		Since \(1-x-2x^2=(1-2x)(1+x)\), this simplifies to
		\[
		\boxed{
			\sum_{n\ge0}T_k(n)x^n
			=
			\frac{x^{k+\delta}(1-x)}
			{(1-2x)^2(1+x)}.
		}
		\]
	\end{proof}
	
	\subsubsection*{Exact closed form for $T_k(n)$}
	
	Set
	\[
	s_k:=k+\delta,
	\]
	so that \(s_k\) is the smallest even integer strictly larger than \(k\).
	From Proposition~\ref{prop:Tk_genfun},
	\[
	\sum_{n\ge0}T_k(n)x^n
	=
	x^{s_k}\frac{1-x}{(1-2x)^2(1+x)}.
	\]
	The remaining rational factor has the elementary decomposition
	\[
	\frac{1-x}{(1-2x)^2(1+x)}
	=
	\frac{1}{3(1-2x)^2}
	+\frac{4}{9(1-2x)}
	+\frac{2}{9(1+x)}.
	\]
	Therefore the coefficients can be written explicitly, with no unspecified
	constants.
	
	\begin{thm}\label{thm:Tk_exact}
		Let \(s_k=k+\delta\). Then
		\[
		T_k(n)=0\qquad (0\le n<s_k),
		\]
		whereas, for \(n\ge s_k\),
		\[
		\boxed{
			T_k(n)
			=
			\frac{(3(n-s_k)+7)\,2^{\,n-s_k}
				+2(-1)^{\,n-s_k}}{9}.
		}
		\]
	\end{thm}
	
	\begin{proof}
		For \(m\ge0\),
		\[
		[x^m]\frac{1}{(1-2x)^2}=(m+1)2^m,\qquad
		[x^m]\frac{1}{1-2x}=2^m,\qquad
		[x^m]\frac{1}{1+x}=(-1)^m.
		\]
		Thus
		\[
		[x^m]\frac{1-x}{(1-2x)^2(1+x)}
		=
		\frac{(m+1)2^m}{3}
		+\frac{4\,2^m}{9}
		+\frac{2(-1)^m}{9}
		=
		\frac{(3m+7)2^m+2(-1)^m}{9}.
		\]
		The factor \(x^{s_k}\) shifts the coefficient by \(s_k\), which gives the
		stated formula.
	\end{proof}

	\begin{rem}
		
		For $n\ge 1$, the total number of compositions of $n$ is $2^{n-1}$. Therefore the
		average number of large even parts among all compositions of $n$ is
		\[
		\mathbb{E}_k(n):=\frac{T_k(n)}{2^{n-1}}.
		\]
	\end{rem}
	
	\medskip
	
	The numerical values are more informative when the three specializations are
	viewed together rather than as isolated large integers. Table~\ref{tab:combined}
	illustrates the dependence on both \(n\) and the threshold \(k\). We also list
	the normalized sign imbalance
	\[
	\Delta_k(n):=\frac{E_k(n)-O_k(n)}{2^{n-1}},
	\]
	which, by Corollary~\ref{cor:arithmetic_sign_imbalance}, is the normalized
	numerical imbalance between the two arithmetic sign classes, up to the factor
	\((-1)^n\).
	
	\begin{table}[ht]
		\centering
		\small
		\begin{tabular}{c|c|r|r|r|r|c|c}
			\(k\)&\(n\)&\(c_k(n)\)&\(E_k(n)\)&\(O_k(n)\)&\(T_k(n)\)&
			\(\mathbb E_k(n)\)&\(\Delta_k(n)\)\\
			\hline
			4  &20&438\,540&441\,960&82\,328&89\,202&0.17014&0.68594\\
			4  &25&13\,280\,482&13\,504\,224&3\,272\,992&3\,728\,270&0.22222&0.60983\\
			4  &30&402\,178\,281&413\,970\,208&122\,900\,704&147\,266\,674&0.27431&0.54216\\
			\hline
			8  &20&520\,079&520\,080&4\,208&4\,210&0.00803&0.98395\\
			8  &25&16\,588\,117&16\,588\,344&188\,872&189\,326&0.01128&0.97748\\
			8  &30&529\,084\,264&529\,103\,680&7\,767\,232&7\,806\,066&0.01454&0.97106\\
			\hline
			12 &20&524\,110&524\,110&178&178&0.00034&0.99932\\
			12 &25&16\,768\,114&16\,768\,114&9\,102&9\,102&0.00054&0.99891\\
			12 &30&536\,470\,425&536\,470\,436&400\,476&400\,498&0.00075&0.99851
		\end{tabular}
		\caption{Joint numerical behaviour of the three principal specializations.}
		\label{tab:combined}
	\end{table}
	
	For fixed \(k\), the average number \(\mathbb E_k(n)\) of large even parts
	increases with \(n\), while the normalized parity imbalance \(\Delta_k(n)\)
	decreases. Conversely, for fixed \(n\), increasing \(k\) makes large even
	parts rarer: \(c_k(n)\) approaches the total number \(2^{n-1}\),
	\(\mathbb E_k(n)\) becomes small, and the even-parity class becomes strongly
	dominant. In particular, the former isolated example \(k=12,n=30\) is now
	seen as part of a systematic trend rather than as a single numerical
	calculation.
	
	\section{Positional statistics: Late appearance of large even parts}\label{sec:positional}
	
	The two positional enumerations below are derived from the same decomposition:
	one first chooses the initial block of parts that avoid \(\mathrm{Ev}_{>k}\),
	and then marks the first admissible large even part.  We therefore emphasize
	the combinatorial decomposition and keep the subsequent algebraic
	simplifications as short as possible.  The purpose of the formulas is to
	quantify precisely the exceptional configurations in which the first
	sign-changing even factor occurs late.
	
	The need for positional information is already visible in the arithmetic
	weight
	\[
	CR_k(m)=\prod_{s=1}^{r}
	\frac{S_{m_s}^{(k)}}{m_1+\cdots+m_s}.
	\]
	Unlike the sign, which is determined by \(B(m)\) alone, the absolute weight
	depends on the successive partial sums. If a large even part \(m_i\) is to be
	modified in an attempt to compare contributions of opposite signs, the
	preceding quantity
	\[
	D=m_1+\cdots+m_{i-1}
	\]
	enters the corresponding denominator estimates. Hence the location of the
	first large even part is not merely an additional refinement of the counting
	problem: it records information that is lost by \(B(m)\) but is relevant to
	the weighted coefficient expansion.
	
	In this section we study a positional statistic: not \emph{how many} large even
	parts occur, but \emph{how late} the first such part appears in the sequence of parts.
	
	Recall that
	\[
	\mathrm{Ev}_{>k}=\{j\ge 1:\ j \text{ even and } j>k\},
	\qquad
	\delta=
	\begin{cases}
		1,& k \text{ odd},\\
		2,& k \text{ even},
	\end{cases}
	\qquad
	\sum_{j\in \mathrm{Ev}_{>k}}x^j=\frac{x^{k+\delta}}{1-x^2}.
	\]
	Let
	\[
	A:=\Z_{>0}\setminus \mathrm{Ev}_{>k}
	\]
	be the set of \emph{allowed} parts that are \emph{not} large even.
	
	\medskip
	
	\paragraph{Generating functions for allowed and forbidden parts.}
	Set
	\[
	U(x):=\sum_{j\ge 1}x^j=\frac{x}{1-x},
	\qquad
	B(x):=\sum_{j\in \mathrm{Ev}_{>k}}x^j=\frac{x^{k+\delta}}{1-x^2},
	\qquad
	S(x):=\sum_{j\in A}x^j=U(x)-B(x).
	\]
	Thus $S(x)$ is the generating function of a single part that is \emph{not} in
	$\mathrm{Ev}_{>k}$.
	
	A direct simplification using $1-x^2=(1-x)(1+x)$ gives the convenient rational form
	\[
	S(x)=\frac{x}{1-x}-\frac{x^{k+\delta}}{1-x^2}
	=\frac{x(1+x)-x^{k+\delta}}{(1-x)(1+x)}.
	\]
	
	\medskip
	
	\paragraph{Compositions with no early large even part.}
	Fix an integer \(\ell\ge0\). We define
	\[
	L_{k,\ell}(n):=
	\#\Bigl\{m=(m_1,\dots,m_r)\in\PP(n):
	m_i\notin\mathrm{Ev}_{>k}\ \text{for every }1\le i\le\min\{\ell,r\}\Bigr\}.
	\]
	Thus \(L_{k,\ell}(n)\) counts compositions for which no large even part occurs
	among the first \(\ell\) positions. This convention also includes compositions
	having fewer than \(\ell\) parts; for such compositions, all existing parts
	must avoid \(\mathrm{Ev}_{>k}\). Equivalently, either the first large even part
	occurs at position at least \(\ell+1\), or no large even part occurs at all.
	
	\begin{prop}\label{prop:Lkl_GF}
		For every fixed \(\ell\ge0\),
		\[
		\boxed{
			\sum_{n\ge0}L_{k,\ell}(n)x^n
			=
			\sum_{r=0}^{\ell-1}S(x)^r
			+
			S(x)^\ell\frac{1}{1-U(x)},
		}
		\]
		where the finite sum is empty when \(\ell=0\).
	\end{prop}
	
	\begin{proof}
		If a composition has \(r<\ell\) parts, all of them must lie in the allowed set
		\(A\), giving the contribution \(S(x)^r\). If it has at least \(\ell\) parts,
		the first \(\ell\) parts must lie in \(A\), contributing \(S(x)^\ell\), while
		the remaining tail is an arbitrary (possibly empty) sequence of unrestricted
		parts, whose generating function is \(1/(1-U(x))\). Summing the two cases gives
		the formula.
	\end{proof}
	
	For later coefficient extraction it is useful to write this as a single
	rational function. Put
	\[
	A_k(x):=x(1+x)-x^{k+\delta},
	\qquad
	D(x):=(1-x)(1+x).
	\]
	Since \(S(x)=A_k(x)/D(x)\) and
	\(1/(1-U(x))=(1-x)/(1-2x)\), one obtains
	\begin{equation}\label{eq:Lkl_rational}
		\sum_{n\ge0}L_{k,\ell}(n)x^n
		=
		\frac{P^{L}_{k,\ell}(x)}
		{(1-2x)(1-x)^\ell(1+x)^\ell},
	\end{equation}
	where
	\[
	P^{L}_{k,\ell}(x)
	=
	(1-2x)\sum_{r=0}^{\ell-1}
	A_k(x)^rD(x)^{\ell-r}
	+
	A_k(x)^\ell(1-x).
	\]
	For \(\ell=0\), this reduces to the classical generating function
	\((1-x)/(1-2x)\) for all compositions.
	
	\medskip
	
	\paragraph{First large even part at position exactly \(\ell+1\).}
	Define, for \(\ell\ge0\),
	\[
	F_{k,\ell}(n):=
	\#\Bigl\{m=(m_1,\dots,m_r)\in\PP(n):
	m_1,\dots,m_\ell\notin\mathrm{Ev}_{>k},
	\quad m_{\ell+1}\in\mathrm{Ev}_{>k}\Bigr\}.
	\]
	Thus \(F_{k,\ell}(n)\) counts compositions whose first large even part occurs
	exactly at position \(\ell+1\).
	
	\begin{prop}\label{prop:Fkl_GF}
		For every fixed \(\ell\ge0\),
		\[
		\sum_{n\ge0}F_{k,\ell}(n)x^n
		=
		S(x)^\ell B(x)\frac{1}{1-U(x)}.
		\]
		Equivalently,
		\begin{equation}\label{eq:Fkl_rational}
			\boxed{
				\sum_{n\ge0}F_{k,\ell}(n)x^n
				=
				\frac{x^{k+\delta}A_k(x)^\ell}
				{(1-2x)(1-x)^\ell(1+x)^{\ell+1}}.
			}
		\end{equation}
	\end{prop}
	
	\begin{proof}
		The first \(\ell\) parts must be allowed, the next part is forced to lie in
		\(\mathrm{Ev}_{>k}\), and the remaining tail is unrestricted. Hence the
		generating function is \(S(x)^\ell B(x)/(1-U(x))\). Substituting the rational
		forms of \(S\), \(B\), and \(1/(1-U)\) gives
		\eqref{eq:Fkl_rational}.
	\end{proof}
	
	\medskip
	
	\paragraph{Late appearance with existence.}
	Because \(L_{k,\ell}(n)\) includes the case in which no large even part occurs,
	while \(c_k(n)\) counts exactly those compositions, we have the exact identity
	\[
	\#\Bigl\{m\in\PP(n):
	\text{the first large even part exists and occurs at position }\ge\ell+1
	\Bigr\}
	=
	L_{k,\ell}(n)-c_k(n).
	\]
	
	\subsection{A coefficient-extraction lemma for the positional statistics}
	
	The two rational functions
	\eqref{eq:Lkl_rational} and \eqref{eq:Fkl_rational} have poles only at
	\(x=\tfrac12,1,-1\), but the poles at \(1\) and \(-1\) generally have higher
	multiplicity. Moreover, the displayed rational functions need not be proper.
	The following elementary lemma records the exact coefficient form in a way
	that handles both issues.
	
	\begin{lem}\label{lem:repeated_poles}
		Let \(a,b\ge0\), let \(P(x)\in\mathbb C[x]\), and set
		\[
		G(x)=\frac{P(x)}
		{(1-2x)(1-x)^a(1+x)^b}.
		\]
		Let \(H(x)\) denote the polynomial part in the Euclidean division of the
		numerator by the denominator. Then, for every \(n\ge0\),
		\[
		\boxed{
			[x^n]G(x)
			=
			[x^n]H(x)
			+
			A\,2^n
			+
			\sum_{j=1}^{a}B_j\binom{n+j-1}{j-1}
			+
			(-1)^n\sum_{j=1}^{b}C_j\binom{n+j-1}{j-1}.
		}
		\]
		The sums are empty when \(a=0\) or \(b=0\). The constants are explicit:
		\[
		A=
		\left.
		(1-2x)G(x)
		\right|_{x=1/2}
		=
		\frac{P(1/2)}
		{(1/2)^a(3/2)^b},
		\]
		and, with \(r=a-j\),
		\[
		B_j
		=
		\frac{(-1)^r}{r!}
		\left.
		\frac{d^r}{dx^r}
		\left(
		\frac{P(x)}{(1-2x)(1+x)^b}
		\right)
		\right|_{x=1},
		\qquad 1\le j\le a,
		\]
		while, with \(r=b-j\),
		\[
		C_j
		=
		\frac{1}{r!}
		\left.
		\frac{d^r}{dx^r}
		\left(
		\frac{P(x)}{(1-2x)(1-x)^a}
		\right)
		\right|_{x=-1},
		\qquad 1\le j\le b.
		\]
		In particular, once \(n>\deg H\), the polynomial contribution
		\([x^n]H(x)\) vanishes.
	\end{lem}
	
	\begin{proof}
		After Euclidean division, the proper rational part has the partial-fraction
		decomposition
		\[
		\frac{A}{1-2x}
		+
		\sum_{j=1}^{a}\frac{B_j}{(1-x)^j}
		+
		\sum_{j=1}^{b}\frac{C_j}{(1+x)^j}.
		\]
		The coefficient formula follows from
		\[
		[x^n](1-2x)^{-1}=2^n,\qquad
		[x^n](1-x)^{-j}=\binom{n+j-1}{j-1},
		\]
		and
		\[
		[x^n](1+x)^{-j}=(-1)^n\binom{n+j-1}{j-1}.
		\]
		The formula for \(A\) is obtained by multiplying by \(1-2x\) and evaluating
		at \(x=\tfrac12\).
		
		For the pole at \(x=1\), multiply the decomposition by \((1-x)^a\).
		The term containing \(B_j\) becomes
		\(B_j(1-x)^{a-j}\). Taking the derivative of order \(r=a-j\) at \(x=1\)
		isolates this coefficient and gives
		\[
		B_j=\frac{(-1)^r}{r!}
		\left[
		\frac{d^r}{dx^r}\bigl((1-x)^aG(x)\bigr)
		\right]_{x=1}.
		\]
		Since \((1-x)^aG(x)=P(x)/((1-2x)(1+x)^b)\), this is the displayed formula.
		The argument at \(x=-1\) is identical, with no alternating derivative sign.
		The polynomial part contributes nothing to these derivatives of orders below
		the pole multiplicities, and it remains explicitly as \([x^n]H(x)\).
	\end{proof}
	
	\subsection{Explicit formulas for \(L_{k,\ell}(n)\) and \(F_{k,\ell}(n)\)}
	
	Applying Lemma~\ref{lem:repeated_poles} to
	\eqref{eq:Lkl_rational} with \(a=b=\ell\) gives the following exact form.
	
	\begin{prop}\label{prop:Lkl_partial_fraction_shape}
		For every \(\ell\ge0\), let \(H^L_{k,\ell}(x)\) be the polynomial part of
		\[
		\frac{P^L_{k,\ell}(x)}
		{(1-2x)(1-x)^\ell(1+x)^\ell}.
		\]
		Then, for every \(n\ge0\),
		\[
		\boxed{
			L_{k,\ell}(n)
			=
			[x^n]H^L_{k,\ell}(x)
			+
			A^L_{k,\ell}2^n
			+
			\sum_{j=1}^{\ell}
			B^L_{k,\ell,j}\binom{n+j-1}{j-1}
			+
			(-1)^n\sum_{j=1}^{\ell}
			C^L_{k,\ell,j}\binom{n+j-1}{j-1},
		}
		\]
		where all constants are given explicitly by the derivative formulas of
		Lemma~\ref{lem:repeated_poles} with \(P=P^L_{k,\ell}\) and \(a=b=\ell\).
		For \(n>\deg H^L_{k,\ell}\), the first term vanishes.
	\end{prop}
	
	Likewise, applying the same lemma to \eqref{eq:Fkl_rational} with
	\(a=\ell\) and \(b=\ell+1\) yields:
	
	\begin{prop}\label{prop:Fkl_partial_fraction_shape}
		For every \(\ell\ge0\), let \(H^F_{k,\ell}(x)\) be the polynomial part of
		\[
		\frac{x^{k+\delta}A_k(x)^\ell}
		{(1-2x)(1-x)^\ell(1+x)^{\ell+1}}.
		\]
		Then, for every \(n\ge0\),
		\[
		\boxed{
			F_{k,\ell}(n)
			=
			[x^n]H^F_{k,\ell}(x)
			+
			A^F_{k,\ell}2^n
			+
			\sum_{j=1}^{\ell}
			B^F_{k,\ell,j}\binom{n+j-1}{j-1}
			+
			(-1)^n\sum_{j=1}^{\ell+1}
			C^F_{k,\ell,j}\binom{n+j-1}{j-1},
		}
		\]
		where all constants are given explicitly by
		Lemma~\ref{lem:repeated_poles} with
		\[
		P(x)=x^{k+\delta}A_k(x)^\ell,\qquad
		a=\ell,\qquad b=\ell+1.
		\]
		For \(n>\deg H^F_{k,\ell}\), the polynomial contribution vanishes.
	\end{prop}
	
	These formulas are exact for all \(n\) and require no assumption of simple
	roots. They are also directly effective: the polynomial parts are obtained by
	Euclidean division, while the remaining constants are finite derivatives at
	the three fixed poles \(1/2,1,-1\). In particular, no numerical root-finding
	is needed.
	
	\begin{exa}
		Take \(k=12\), so \(\delta=2\), and \(\ell=5\). Then
		\[
		\sum_{n\ge0}F_{12,5}(n)x^n
		=
		\frac{x^{14}\bigl(x(1+x)-x^{14}\bigr)^5}
		{(1-2x)(1-x)^5(1+x)^6}.
		\]
		For the late-or-none statistic, the corrected generating function is
		\[
		\sum_{n\ge0}L_{12,5}(n)x^n
		=
		\sum_{r=0}^{4}
		\left(
		\frac{x(1+x)-x^{14}}{(1-x)(1+x)}
		\right)^r
		+
		\left(
		\frac{x(1+x)-x^{14}}{(1-x)(1+x)}
		\right)^5
		\frac{1-x}{1-2x}.
		\]
		Coefficient extraction gives
		\[
		L_{12,5}(30)=536\,652\,932,
		\qquad
		F_{12,5}(30)=42\,274.
		\]
		Since \(c_{12}(30)=536\,470\,425\), the number of compositions of \(30\)
		for which a large even part exists and its first occurrence is at position at
		least \(6\) is
		\[
		L_{12,5}(30)-c_{12}(30)=182\,507.
		\]
	\end{exa}

	\section*{Conclusion and perspectives}

	In this paper, we developed a systematic generating function approach to study
	integer compositions in which the occurrence of even parts larger than a fixed
	threshold $k$ is controlled. This framework leads to explicit rational generating
	functions, linear recurrences, and exact closed formulas for several refined
	statistics, including the number of such parts, their parity, their total
	contribution, and the position of their first occurrence in a composition.
	
	One of the strengths of the method is its flexibility. The same strategy applies
	to many other families of parts in integer compositions, leading again to
	rational generating functions and to explicit enumeration formulas and
	asymptotic information. In this sense, the present work provides a general and
	flexible combinatorial framework for the study of weighted or constrained
	compositions.
	
	The four families of quantities considered here can therefore be read as
	successive combinatorial shadows of the same arithmetic problem: \(c_k(n)\)
	measures the complete absence of sign-changing large even parts,
	\(E_k(n)-O_k(n)\) measures the numerical imbalance between the two sign
	classes, \(T_k(n)\) measures the total frequency of the sign-changing factors,
	and the positional statistics retain information relevant to the partial-sum
	denominators in the weights \(CR_k(m)\). Their generating functions admit exact
	coefficient formulas involving only the fixed poles \(1/2,1,-1\), including the
	necessary polynomial terms when the rational functions are improper. The
	weighted comparison itself is a separate arithmetic problem and is deliberately
	not claimed here.
	
	As explained in the introduction, our original motivation comes from arithmetic
	questions related to zeta functions and the coefficients of $L$-polynomials of
	algebraic curves over finite fields, in particular in the context of towers of
	function fields. Rather than providing merely another route to a sign
	determination, the present approach gives finer quantitative information on the
	combinatorial structure of the terms occurring in the coefficient expansion:
	it records the absence, parity, multiplicity, and position of the large even
	parts responsible for sign changes. Such information may also be useful in the
	study of other towers whose arithmetic data exhibit a comparable sign pattern,
	and suggests a broader role for these generating-function methods in the
	analysis of \(L\)-polynomial coefficients.

	\section*{Acknowledgements}
	
	The author would like to thank Stéphane Ballet for many helpful discussions. 
	This work unexpectedly emerged from his ideas in algebraic geometry.

	\bibliographystyle{plain}
	\bibliography{Biblio_Den}

\end{document}